\documentclass[11pt]{amsart}
\usepackage{etex}
\makeatletter
\def\alloc@#1#2#3#4#5%
{\ifnum\count1#1<#4% make sure there's still room
	\allocationnumber\count1#1
	\global\advance\count1#1\@ne
	\global#3#5\allocationnumber
	\wlog{\string#5=\string#2\the\allocationnumber}%
	\else\ifnum#1<6
	\def\etex@dummy@definition{}% <-- code added
	\begingroup \escapechar\m@ne
	\expandafter\alloc@@\expandafter{\string#2}#5%
	\else\errmessage{No room for a new #2}\fi\fi
}
\makeatother
\usepackage[thinlines,thiklines]{easybmat}
\usepackage{amsfonts}

\usepackage[all,cmtip]{xy}
\usepackage{amsmath}
\usepackage{amsthm}
\usepackage{amssymb}
\usepackage{enumitem}
\usepackage{booktabs}
\usepackage{tikz}
\usepackage{longtable}
\usetikzlibrary{calc,arrows,backgrounds}
\usepackage{float}
\usepackage{floatflt}
\usepackage[left=1.5in,top=1.2in,right=1.5in,bottom=1.2in]{geometry}

\usepackage[backref=page]{hyperref}
\usepackage{hyperref}
\hypersetup{pdftitle={Gap rigidity}}
\hypersetup{pdfauthor={Cong Ding}}
\hypersetup{colorlinks=true,linkcolor=blue,anchorcolor=blue,citecolor=blue}

\newtheorem{theorem}{\bf Theorem}[section]
\newtheorem{definition}[theorem]{\bf Definition}
\newtheorem{corollary}[theorem]{\bf Corollary}
\newtheorem{lemma}[theorem]{\bf Lemma}
\newtheorem{proposition}[theorem]{\bf Proposition}

\newtheorem{remark}[theorem]{\bf Remark}

\newtheorem*{theorem*}{Theorem}
\def\End{{\rm End}}
\begin{document}

\title[Gap rigidity]{On Gap rigidity problems for compact Hermitian symmetric spaces}
\author{Cong Ding}
	\address{Morningside Center of Mathematics, Academy of Mathematics \& Systems Science, The Chinese Academy of Sciences, Beijing, China}
	\email{congding@amss.ac.cn}
	\maketitle
	
\begin{abstract}
	We prove a gap rigidity theorem for diagonal curves in maximal polyspheres of irreducible compact Hermitian symmetric spaces of tube type, which is a dual analogy to a theorem obtained by Mok. Motivated by the proof we show that Chow space of certain totally geodesic submanifolds is affine algebraic, which gives a weaker version of gap rigidity for a class of higher dimensional submanifolds.
\end{abstract}




\section{Introduction}
The gap rigidity problem for locally Hermitian symmetric space of noncompact type has been raised  by Mok\cite{MR1918134} and Eyssidieux-Mok \cite{MR2127948}. Write a pair of bounded symmetric domains as $i:D\hookrightarrow\Omega$ where $i$ is a geodesic embedding. Write $\Omega=G_0/K$ where $G_0=Aut(\Omega)$ and $K$ is the isotropy subgroup with respect to a reference point $o$, there is a $K^{\mathbb{C}}$-invariant Zariski open subset $\mathcal{O}_o$ in $Gr(\dim(D),T_o(\Omega))$ such that $[T_o(D)]\in \mathcal{O}_o$. If $\Gamma$ is a torsion free discrete subgoup of automorphisms, we say the gap rigidity holds in the Zariski topology if for any compact complex submanifold $S$ in the quotient $\Omega/\Gamma$ with $[T_x(S)]\in \mathcal{O}_o$ (by some lifting) at every point $x\in S$, $S$ must be totally geodesic.

\par In \cite{MR1918134}, a typical example is given using Poincar\'e-Lelong equation where $\Omega$ is an irreducible bounded symmetric domain of tube type with rank $\geq 2$ and $D$ is a diagonal curve in a maximal polydisk, and in this case $\mathcal{O}_o$ is given by the complement of highest characteristic subvariety in $\mathbb{P}(T_o(\Omega))$ (see \cite{MR1918134} for details), which is a hypersurface when $\Omega$ is of tube type. More precisely they obtained

\begin{theorem}\label{thm polydisk}\rm(\textbf{Mok, \cite[Theorem~1]{MR1918134}})
Suppose that $\Omega$ is an irreducible bounded symmetric domain of tube type with rank $\geq 2$ and $\Gamma$ is a torsion free discrete subgoup of automorphisms. If $C$ is a compact smooth curve in $\Omega/\Gamma$ such that every tangent vector on $C$ is generic, then $C$ is totally geodesic in $\Omega/\Gamma$. 
\end{theorem}
By using a similar method on higher dimensional submanifolds, Eyssidieux-Mok \cite{MR2127948} generalized the result to any ($H_3$)-holomorphic geodesic cycle, where $H_3$ is defined in \cite[p.10]{MR2127948}, as a condition stronger than total geodesy. The theorem can be stated as follows.
\begin{theorem}\label{EyssidieuxMok}\rm(\textbf{Eyssidieux-Mok, \cite[Theorem~3]{MR2127948}})
Suppose that $\Omega=G_0/K$ is an irreducible bounded symmetric domain and $\Gamma$ is a torsion free discrete subgoup of automorphisms, $D\subset\Omega$ is an ($H_3$)-embedding and there is a $K^{\mathbb{C}}$-invariant hypersurface $\mathcal{Z}_o$ in $Gr(\dim(D),T_o(\Omega))$ such that $[T_o(D)]\notin \mathcal{Z}_o$, if $S\subset \Omega/\Gamma$ is a compact complex submanifold with $[T_x(S)]\notin \mathcal{Z}_x$ at every $x\in S$, then $S$ is an $(H_3)$-holomorphic geodesic cycle.
\end{theorem}

\par The gap rigidity problem was originally studied by Eyssidieux-Mok \cite{MR1369135} in complex topology, i.e. in differential geometric sense, which is a weaker form of gap phenomenon. More precisely, the gap rigidity holds in the sense of complex topology for a pair of bounded symmetric domains $(\Omega, D;i)$
if a compact complex submanifold $S\subset \Omega/\Gamma$ modelled on $(\Omega, D;i)$ (roughly speaking, locally approximated by $i(D)$) with uniformly sufficiently small norm of the second fundamental form is necessarily totally geodesic. Simple example like diagonal embedding of disk into polydisk was considered.
Examples for gap rigidity concerning period domains from Hodge Theory were studied in \cite{MR1468929}\cite{MR1714821}.

\par The dual analogy of the gap rigidity problem can be easily formualted. Let $M=G/P$ be an irreducible compact Hermitian symmetric space where $G$ is a connected complex simple Lie group and $P$ is a maximal parabolic subgroup. Assume that $i:X\hookrightarrow M$ is an equivariant geodesic embedding and there is a $P$-invariant Zariski open subset $\mathcal{O}_o\subset Gr(\dim(X),T_o(M))$ such that $[T_o(X)]\in \mathcal{O}_o$, we say gap rigidity holds for the pair of Hermitian symmetric spaces $(M,X,i)$ (in Zariski topology) if for any compact complex submanifold $S$ with tangent space $[T_x(S)]$ lifting to an element in $\mathcal{O}_o$ for every point $x\in S$,  $S$ must be some standard model of $X$ where the standard models of $X$ are defined as $g\circ i(X) \subset M$ for all $g\in G$.

\par  In this article,
we are going to prove that the gap rigidity holds for $(M,X,i)$ where $M$ is an irreducible compact Hermitian symmetric space of tube type and $X$ is a diagonal curve in a maximal polysphere, giving a dual analogy for Theorem 1 in \cite{MR1918134}. We have 
\begin{theorem}\label{diagonal curve}
Let $M$ be an irreducible compact Hermitian symmetric space of tube type with $rank(M)\geq 2$ and $C\subset M$ is a compact curve with generic tangent vectors at every point, i.e. at every point $x\in C$ there are standard models of diagonal curves in maximal polyspheres tangent to $C$ at $x$, then $C$ itself must be a standard model.
\end{theorem}
The theorem is not true in general when $M$ is a non-tube type irreducible compact Hermitian symmetric space. For instance, for the Grassmannian $G(2,3)$ (denoting all 2-planes in a 5-dimensional complex vector space), there is a standard embedding from $\mathbb{P}^1\times \mathbb{P}^2$ to $G(2,3)$. We take the graph of the Veronese embedding from $\mathbb{P}^1$ to $\mathbb{P}^2$, then this gives a compact curve with generic (rank two) tangent vectors which is not a standard model of the diagonal curve. 

Our proof is to construct a holomorphic map from the compact submanifold $S$ to the Chow space of standard models and show that the Chow space is affine algebraic. For each point $x$ on $S$, there is a family of standard models tangent to $S$ to the first order, we only need to show that there is a unique standard model tangent to $S$ to the second order, then the holomorphic map can be constructed. The affine-algebraicity comes from the fact that the isotropy group of the Chow space is reductive by a double fibration construction and dimension counting argument.

\par Motivated by the proof, we find  the affine-algebraicity of the Chow space of certain higher dimensional submanifolds satisfying the so-called $(H_2)$-condition (stronger than total geodesy), we obtain
\begin{theorem}\label{generalaffineness}
Suppose that $M=G/P$ is an irreducible compact Hermitian symmetric space with rank $\geq 2$. $X\subset M$ is a totally geodesic equivariant $(H_2)$-subspace with respect to some choice of canonical K\"ahler-Einstein metric on $M$, where  $\dim(M)=n,\dim(X)=m$. And there is a $P$-invariant hypersurface $\mathcal{Z}_o$ in $Gr(m,T_o(M))$ such that $P.[T_o(X)]\subset \mathcal{O}_o:=Gr(m,T_o(M))-\mathcal{Z}_o$. Then Chow space of the standard models, i.e. the $G$-orbit $G.[X]\subset Chow(M)$ is affine algebraic.
\end{theorem}
This can be considered as a contrary to the Chow space of a given smooth Schubert variety in $M$, which is projective algebraic. In curve case as an example, a smooth Schubert variety is a rational curve of minimal degree with respect to the ample generator of $Pic(M)$ which has rank one tangent vectors. While in our situation the diagonal curves have tangent vectors of maximal rank. These two 'extreme cases' have such interesting properties.

We then give a weaker version of gap rigidity for such $(H_2)$-subspaces.
\begin{corollary}\label{general}
Under the setting in Theorem \ref{generalaffineness}, if $S\subset M$ is a compact complex submanifold and by some lifting $[T_x(S)] \in P.[T_o(X)]$ at every $x\in S$, and moreover at every point $x\in S$ there are standard models of $X$ tangent to $S$ to the second order at $x$, then $S$ itself is a standard model.

\end{corollary}

\begin{remark}\label{reducible}
$(H_2)$-condition is required to satisfy a Lie bracket generating condition (see Section \ref{weakergap}). 
We know if $X$ is $H_2$ in $M=M_1\times \cdots \times M_k$, then $X$ is $H_2$ in each factor $M_i (1\leq i\leq k)$. 
For convenience we assume that $M=G/P$ is irreducible in our  discussion in this article. In fact the case when $M$ is reducible can be also proved under similar formulation following similar argument.
\end{remark}

\begin{remark}
Corollary \ref{general} is a weaker version of gap rigidity. There is a restriction on the tangent spaces of $S$, they are required to be tangent to some standard models to the second order, not just contained in the Zariski open set $\mathcal{O}_o \subset Gr(m,T_o(M))$. If $P.[T_o(X)]$ is coincident with $\mathcal{O}_o$ and the second order condition is automatically satisfied as the case where we consider diagonal curve in a tube type ambient space, then it is the same as general gap rigidity. And a simple example satisfying the (weaker) gap rigidity for reducible ambient space is the diagonal embedding of a Riemann sphere $X=\mathbb{P}^1$ in a polysphere $M=(\mathbb{P}^1)^r$, where the $P$-action on $T_o(M)$ is actually the $\mathbb{C}^*$-action on tangent vectors of  each sphere and the $P$-invariant hypersurface in the projectivized tangent space can be chosen as a hyperplane.
\end{remark}

\begin{remark}
We are still curious about whether the second order tangency can be removed in the condition of the Corollary \ref{general}. When $X$ is a curve, we will find that the second order tangency is an empty condition by our proof, but for general case it is still not known. Since we consider the submanifold $S$ globally, the second order condition may be implicitly contained in the global condition. In fact we will find that when $M$ is a hyperquadric and $X$ is a smooth linear section subquadric, the second order tangency is actually redundant. As in this case if a compact submanifold $S$ is tangent to the some standard models of smooth linear section subquadrics only to the first order at every point, then $S$ is a submanifold with splitting tangent sequence in $M$ where the direct sum $T(M)=T(S)\oplus N_{o,S|M}$ is given by the annihilator with respect to the holomorphic conformal structure of the hyperquadric. Therefore by the result of Jahnke (cf. \cite[Theorem~4.7]{MR2190340}) we know $S$ itself must be a smooth linear section subquadric. From this example we may conjecture that the second order tangency is not necessary, but we have no more evidence until now.
\end{remark}
We know there are four classical types and two exceptional types ($E_6, E_7$ types) for irreducible compact Hermitian symmetric spaces. In this article we use the following notations for the classical types. $G(p,q)$ (Type I) denotes the Grassmannian of $p$-dimensional subspace in a $(p+q)$-dimensional complex vector spaces; $G^{II}(n,n)$ (Type II) and $G^{III}(n,n)$ (Type III) denote the orthogonal Grassmannian and Lagrangian Grassmannian respectively;  $Q^n$ (Type IV) denotes the hyperquadric. The article is organized as follows. In Section \ref{prel} we give some preliminaries on Hermitian symmetric spaces and $(H_2)$-subspaces. In Section \ref{gapforcurves} we give a proof of Theorem \ref{diagonal curve}. In Section \ref{weakergap} we give a proof of Theorem \ref{generalaffineness} and Corollary \ref{general}.

\section{Preliminaries}\label{prel}
\subsection{Hermitian symmetric spaces}
We write an irreducible compact Hermitian symmetric space of rank $r\geq 2$ as $M=G/P=G_c/K$ where $G$ is a connected complex simple Lie group, $P$ is a maximal parabolic subgroup, $G_c$ is the isometry group of $M$ with respect to a canonical choice of K\"{a}hler-Einstein metric on $M$ which is a compact real form of $G$, $K$ is a maximal compact subgroup of $G_c$. Let $\mathfrak{g}=Lie(G)$, then the Harish-Chandra decomposition can be written as $\mathfrak{g}=\mathfrak{m}^++\mathfrak{l}^\mathbb{C}+\mathfrak{m}^-$ where $\mathfrak{l}=Lie(K)$ and $\mathfrak{l}^\mathbb{C}+\mathfrak{m}^-=Lie(P)$. The holomorphic tangent space at a reference point $o$ is $T_{o}(M)\cong\mathfrak{m}^+$. Considering the isotropy action of $P$ on the projectivized tangent space $\mathbb{P}T_o(M)$, since we have Levi decomposition $P=K^{\mathbb{C}}M^-$ where $M^-$ is the unipotent radical and acts trivially on $T_o(M)$, $P$-orbits on $\mathbb{P}T_o(M)$ are the same as $K^{\mathbb{C}}$-orbits and similarly they are the same on $Gr(p,T_o(M))$ for $p<\dim(M)$ which denotes all $p$-planes in the tangent space $T_o(M)$. It is well-known that there are precisely $r$ $K^{\mathbb{C}}$-orbits on $\mathbb{P}T_o(M)$ according to the rank of the vector $\eta\in T_o(M)$, which actually comes from the Polydisk and Polysphere Theorem and the Restricted Root Theorem (see Theorem \ref{RRR}).
\begin{theorem}[The Polydisk and Polysphere Theorem]
Suppose that $(X_0,g_0)$ is a rank $r$ irreducible Hermitian symmetric space of noncompact type, $X_0=G_0/K$. Then there exists a totally geodesic complex submanifold $(D,g_0|_D)$ isometric to a a product of $r$ disks equipped with Poincar\'{e} metric and the $K$-action on $D$ exhausts $X_0$, i.e. $X_0=\bigcup_{k\in K}kD$. And for the dual case $(X_c, g_c)$ there exists a totally geodesic submanifold $S$ in $X_c$ isometric a a product of $r$ Riemann spheres equipped with Fubini-Study metric. Moreover $D$ is contained in $S$ through the Borel embedding of $X_0$ in $X_c$. 
\end{theorem}

\par More precisely we fix a Cartan subalgebra $\mathfrak{h}^{\mathbb{C}}$ of $\mathfrak{l}^{\mathbb{C}}$ and give a root system $\Delta\subset( \mathfrak{h}^{\mathbb{C}})^*$ for $\mathfrak{g}$. Let $z$ be a central element of $\mathfrak{l}$ such that $ad(z)$ gives the almost complex structure of $M$, and let $iy\in i\mathfrak{h}$ be in the interior of a Weyl chamber whose closure contains $iz$, then we can define a positive root system $\Delta^+=\{\alpha\in \Delta:\alpha(iy)>0\}$ and we denote the set of negative roots by $\Delta^-$ correspondingly. Then we can write $\mathfrak{m}^+=\sum_{\alpha \in \Delta^+_{M}}\mathfrak{g}_{\alpha}$, $\mathfrak{m}^-=\sum_{\alpha \in \Delta^-_{M}}\mathfrak{g}_{\alpha}$ where $\mathfrak{g}_{\alpha}$ is the corresponding root space. The roots in the set $\Delta_{M}^+(\Delta_{M}^-)$ are called noncompact positive(negative) roots, and the other roots in $\Delta^+(\Delta^-)$ are called compact positive(negative) roots, for which we use $\Delta^{+}_K(\Delta^-_K)$ to denote the set of them. For each root $\alpha$,  $e_{\alpha}$ denotes the normalized root vector of $\alpha$ such that $[e_{\alpha},e_{-\alpha}]=H_{\alpha}$ where $H_{\alpha}$ is the dual element of $\alpha$ in $\mathfrak{h}^{\mathbb{C}}$ with respect to the Killing form of $\mathfrak{g}$, and let $h_\alpha$ be the corresponding coroot. Recall that two roots $\alpha, \beta$ are said to be strongly orthogonal to each other if $\alpha\pm\beta$ are not roots.
There is a maximal set of strongly orthogonal positive noncompact roots $\Pi=\{\alpha_1,...\alpha_r\}\subset \Delta^+_M$ constructed from the highest root. Then each representative for the $K^{\mathbb{C}}$-action is $\sum_{i=1}^pe_{\alpha_i}$ for $1\leq p\leq r$. For a nonzero holomorphic tangent vector $\eta\in T_o(M)$, if $Adk.\eta=\sum_{i=1}^pe_{\alpha_i}$ for some $k\in K^{\mathbb{C}}, 1\leq p\leq r$, then $p$ is called the rank of $\eta$.

\par  We say a vector $\eta\in T_o(M)$ is generic if $\eta$ is of rank $r$. In our case, each tangent vector of a diagonal curve in a maximal polysphere is generic. We say an irreducible bounded symmetric domain is of tube type if it is biholomorphic to a tube domain. By case-by-case study, in \cite{MR1918134} they observed that $K^{\mathbb{C}}$-orbit of a maximal rank vector is equal to the complement of a hypersurface (the so-called highest characteristic subvariety) in $\mathbb{P}T_o(M)$ if and only if $M$ is dual to an irreducible bounded symmetric domain of tube type, i.e. it is one of the following.
\begin{enumerate}
\item $G(n,n)$ with $n\geq 2$;
\item $G^{II}(n,n)$ with $n$ even and $n\geq 4$;
\item  $G^{III}(n,n)$ with $n\geq 2$;
\item  $Q^n$ with $n\geq 3$;
\item $E_7$-type.
\end{enumerate}
We will call them irreducible compact Hermitian symmetric spaces of tube type. The Restricted Root Theorem (cf.\cite{MR0161943}) gives a description on the roots restricted to the real vector space $\mathfrak{h}^-=\sum_{\alpha\in \Pi}\sqrt{-1}H_{\alpha}\mathbb{R}$. corresponding to tube type and non-tube type Hermitian symmetric spaces.
\begin{theorem}[The Restricted Root Theorem]\label{RRR}
Let $\rho$ denote the restriction of roots from $\mathfrak{h}^{\mathbb{C}}$ to $\mathfrak{h}^-$, identify the elements in $\Delta$ with their $\rho$-image, then either $\rho(\Delta)\cup\{0\}=\{\pm\frac{1}{2}\alpha_i\pm\frac{1}{2}\alpha_j:1\leq i,j\leq r\}$ or $\rho(\Delta)\cup\{0\}=\{\pm\frac{1}{2}\alpha_i\pm\frac{1}{2}\alpha_j, \pm\frac{1}{2}\alpha_i:1\leq i,j\leq r\}$. Accordingly $\rho(\Delta^+_M)=\{\frac{1}{2}\alpha_i+\frac{1}{2}\alpha_j:1\leq i,j\leq r\}$ or $\rho(\Delta^+_M)=\{\frac{1}{2}\alpha_i+\frac{1}{2}\alpha_j, \frac{1}{2}\alpha_i:1\leq i,j\leq r\}$. Moreover, all $\alpha_i$ have the same length and the subgroup of the Weyl group of $G$ preserving the compact roots and fixing $\Pi$ as a set induces all signed permutations $\alpha_i\rightarrow \pm\alpha_j$ of $\Pi$.
\end{theorem}
From the perspective of Restricted Root Theorem, $M$ is of tube type if and only if it is the first case in the theorem.
\subsection{$(H_2)$-embeddings}
The readers may consult \cite{MR0196134} for more details. Suppose we have a pair of semisimple Lie algebras of Hermitian type $(\mathfrak{g}_c, H_0)$, $(\mathfrak{g}_c', H'_0)$ corresponding to two Hermitian symmetric spaces of compact type $G_c/K$, $G_c'/K'$ respectively, where $H_0, H'_0$ are correponding to the central element in the associated maximal compact subgroups $K,K'$ determining the complex structure of $G_c/K$ and $G_c'/K'$ respectively, we call a Lie algebra homomorphism $\rho :\mathfrak{g} \longrightarrow \mathfrak{g}'$ satisfying $\rho\circ ad(H_0)=ad(H'_0)\circ\rho$ an $(H_1)$-homomorphism, injecitive $(H_1)$-homomorphism are one-to-one corresponding to totally geodesic complex submanifolds which are called $(H_1)$-subspaces.
Furthermore if we have $\rho(H_0)=H'_0$, the Lie algebra homomorphism is called $(H_2)$-homomorphsim and injective $(H_2)$-homomorphsims give so-called $(H_2)$-subspaces. There are some simple examples which are $H_1$ but not $H_2$, like $\mathbb{P}^1\times \mathbb{P}^2 \subset G(2,3)$ by standard equivariant embedding. The full classification of $(H_2)$-subspaces in Hermitian symmetric spaces are given in \cite{MR0196134} and \cite{MR0214807}.
In this article we will refer the reader to the table given in \cite[pp.27-28]{MR2127948} for maximal $(H_2)$-subspaces in irreducible Hermitian symmetric spaces. For readers' convenience, we give the table when $M$ is of classical type (with some corrections)
\begin{longtable}{|p{4em}|p{10em}|p{5em}|p{10em}|}
\caption{
	Maximal $(H_2)$-subspaces $X$ of a compact irreducible Hermitian symmetric space of classical type $M$} \\
\hline
$M$  & $X\subset M$  & maximal & Additional conditions 
\\ \hline
$G(p,q)$ & $G(r,s)\times G(p-r,q-s)$ & * & $\frac{r}{s}=\frac{p}{q}$ \\
\hline 
& $G^{II}(n,n)$  &  * & $p=q=n$ \\
\hline  
& $G^{III}(n,n)$  &  * & $p=q=n$ \\
\hline  
& $\mathbb{P}^m$  & $m\equiv 0[2]$ & $p=\binom{m}{r-1}, q=\binom{m}{r}, r\in \mathbb{N}$ \\
\hline  
& $Q^{2\ell}$  &  $\ell \equiv 0[2]$ & $p=q=2^{\ell-1}, \ell\geq 3$ 
\\ \hline

$G^{II}(n,n)$  &$G(r,r)$  & * &$n=2r$
\\ \hline
& $G^{II}(r,r)\times G^{II}(n-r,n-r)$ & * & $n>r$ \\
\hline 
& $\mathbb{P}^m$  & * & $n=\binom{m}{\frac{m+1}{2}},m\equiv 3[4]$ \\
\hline  
& $Q^{2\ell}$  &  * &   
$n=2^{\ell-1}, \ell\equiv 3[4], \ell\geq 3$ \\
\hline  
& $Q^{2\ell-1}$  &  * &   
$n=2^{\ell-1}, \ell\equiv 0, 3[4], \ell\geq 3$ \\  \hline

$G^{III}(n,n)$  &$G(r,r)$  & * &$n=2r$\\
\hline  
& $G^{III}(r,r)\times G^{III}(n-r,n-r)$ & * & $n>r$ \\
\hline 
& $\mathbb{P}^m$  & * & $n=\binom{m}{\frac{m+1}{2}},m\equiv 1[4]$ \\
\hline  
& $Q^{2\ell}$  &  * &   
$n=2^{\ell-1}, \ell\equiv 1[4], \ell\geq 3$ \\
\hline  
& $Q^{2\ell-1}$  &  * &   
$n=2^{\ell-1}, \ell\equiv 1, 2[4], \ell\geq 3$ \\
\hline  

$Q^{2\ell}$ &  $Q^{2\ell-1}$ & *& $\ell \geq 3$\\
\hline 

$Q^{2\ell-1}$ &  $Q^{2\ell-2}$ & *& $\ell \geq 3$ \\
\hline
\end{longtable}

\section{Gap rigidity for diagonal curves}\label{gapforcurves}
\par In this section we give the proof of Theorem \ref{diagonal curve}. Let $M=G/P=G_c/K$ be an irreducible compact Hermitian symmetric space of tube type with rank $r\geq 2$ and $C\subset M$ is a compact curve with generic (rank $r$) tangent vectors at every point. Our proof is to construct a holomorphic map from the compact curve $C$ to the Chow space of standard models of diagonal curves in maximal polyspheres and show that the Chow space is affine algebraic, then the map should be constant and $C$ itself must be some standard model. For each point $x$ on $C$, there is a family of standard models tangent to $C$ to the first order, we will show that there is a unique standard model tangent to $C$ to the second order, which gives the holomorphic map. The affine-algebraicity comes from the fact that the isotropy group of the Chow space is reductive by a double fibration construction and dimension counting argument. 

\subsection{Construction of a holomorphic map from the curve to the Chow space of standard models.}

We will give a proof for the existence and uniqueness of a standard model which is tangent to $C$ at one point $o$ to the second order. Let $M^-=\exp(\mathfrak{m}^-)$, the action given by $M^-$ will transform the standard model fixing the tangent space at $o$. 
For $\eta\in T_{o}(M)\cong \mathfrak{m}^+$,  we know the curve $\eta(t)=t\eta\subset \mathfrak{m}^+$. Fix $\xi\in \mathfrak{m}^-$, under the $M^-$-action $\exp(\xi)$, the second order term of the curve is given by 
$ \exp\left(\frac{t^2}{2}[\eta,[\xi,\eta]]\right)\cdot o$
(cf. \cite[Lemma 4.3]{MR1198602}). Also the second fundamental form of the standard model after the $M^-$-action is given by \[\sigma(\eta,\eta)=[\eta,[\xi,\eta]] \mod \mathbb{C}\eta\]

Now we fix a root system and a maximal set of strongly orthogonal positive noncompact roots $\Pi=\{\alpha_1,...\alpha_r\}$. The normalized root vectors are denoted by $e_{\alpha_i}$. Considering the case when $\eta=\sum_{i=1}^re_{\alpha_i}$, we can prove the following lemma 
\begin{lemma}{\label{bijection}}
The linear map $H:\mathfrak{m}^- \rightarrow \mathfrak{m}^+$ given by \[H(v)=[\sum_{i=1}^re_{\alpha_i},[v,\sum_{i=1}^re_{\alpha_i}]]\] is a bijection.
\end{lemma}
\begin{proof}
We can check case-by-case. $M$ can be classified as (1) $G(n,n)$ with $n\geq 2$; (2) $G^{II}(n,n)$ with $n$ even and $n\geq 4$; (3) $G^{III}(n,n)$ with $n\geq 2$; (4) $Q^n$ with $n\geq 3$; or (5) $E_7$-type irreducible compact Hermitian symmetric space.
\par For the cases (1),(2) and (3), it is straightforward by matrix computation. Write $\sum_{i=1}^re_{\alpha_i}$ as 
$\begin{bmatrix}
0 & D \\
0 & 0
\end{bmatrix}$ and $v$ as $\begin{bmatrix}
0 & 0 \\
A & 0
\end{bmatrix}$
Then $H(v)$ can be written as 
$\begin{bmatrix}
0 & 2DAD \\
0 & 0
\end{bmatrix}$.
For $G(n,n)$ or $G^{III}(n,n)$ with $n\geq 2$, $r=n$, $D=I_r$ where $I_r$ is the $r\times r$ identity matrix. For $G^{II}(n,n)$ with $n$ even and $n\geq 4$, $r=\frac{n}{2}$,  we have 
$
D=J_r=diag(J_1,...J_1)
$ where $J_1=\begin{bmatrix}
0&1\\
-1&0
\end{bmatrix}$.
Hence for these cases $H$ is clearly a bijection. \par As hyperquadric case has already been proven in \cite[Proposition 2.3]{Zhang2015}, we remain to check the $E_7$ case, the strongly orthogonal roots are written as $\Pi=\{\alpha_1,\alpha_2, \alpha_3\}$. It suffices to show that $H$ is a surjection, i.e. for any root vector $e_{\gamma}\in \mathfrak{m}^+$, there exists $v\in \mathfrak{m}^-$, such that $[e_{\alpha_1}+e_{\alpha_2}+e_{\alpha_3},[v,e_{\alpha_1}+e_{\alpha_2}+e_{\alpha_3}]]=ce_{\gamma}$ for some nonzero constant $c$. If $\gamma=\alpha_i\in \Pi$, we can just choose $v=e_{-\alpha_i}$.
If $\gamma\notin \Pi$, for simplicity we introduce 
\begin{definition}
	For any postive noncompact root $\gamma$ we will call a triple of positive noncompact roots $(\alpha_i,\alpha_j,\beta)$ a \textbf{compatible triple for} 
	$\mathbf{\gamma}$ if $\beta=\alpha_i+\alpha_j-\gamma$ and  $\alpha_i-\gamma, \alpha_j-\gamma$ are roots, where $\alpha_i,\alpha_j\in \Pi$.
\end{definition} 

Then it suffices to show that for any positive noncompact root $\gamma\notin \Pi$, there exists a compatible triple $(\alpha_i, \alpha_j, \beta)$ for $\gamma$ and also there is no $\alpha_{i'},\alpha_{j'}\in \Pi$ such that $(\alpha_{i'},\alpha_{j'},\beta)$ is a compatible triple for another positive noncompace root $\gamma'\neq \gamma$. Then $v$ can be chosen as $e_{-(\alpha_i+\alpha_j-\gamma)}$. This property is straightforward to check when $M$ is classical.

\par We now check the root system of $E_7$. Let $x_i (1\leq i\leq 7)$ be the standard basis of $\mathbb{R}^7$. The positive roots are $x_i-x_j(1\leq i<j\leq 7)$, $x_i+x_j+x_k(1\leq i<j<k\leq 7)$ and $d-x_i(1\leq i\leq 7)$ where $d=\sum_{i=1}^7x_i$. The positive noncompact roots  are $x_1-x_i(2\leq i\leq  7)$, $x_1+x_i+x_j(2\leq i<j\leq 7)$ and $d-x_i(2\leq i\leq 7)$. The maximal strongly orthogonal roots can be chosen as $\alpha_1=x_1-x_2, \alpha_2=x_1+x_2+x_3, \alpha_3=d-x_3$. For $\gamma=x_1-x_3$, we have the triple $(\alpha_1,\alpha_3,d-x_2)$; For $\gamma=x_1-x_i$ with $4\leq i\leq 7$, we have the triple $(\alpha_1,\alpha_2,x_1+x_3+x_i)$; For $\gamma=x_1+x_2+x_j$ with $4\leq j\leq 7$, we have the triple $(\alpha_2,\alpha_3,d-x_j)$; For $\gamma=x_1+x_3+x_j$ with $4\leq j\leq 7$, we have the triple $(\alpha_1,\alpha_2,x_1-x_j)$; For $\gamma=x_1+x_i+x_j$ with $4\leq i<j\leq 7$, we have the triple $(\alpha_1,\alpha_3,x_1+x_k+x_\ell)$ with $\{k,\ell\} =\{4,5,6,7\}-\{i,j\}$. For $\gamma=d-x_2$, we have the triple  $(\alpha_1,\alpha_3,x_1-x_3)$; For $\gamma=d-x_i$ with $4\leq i\leq 7$ we have the triple  $(\alpha_2,\alpha_3,x_1+x_2+x_i)$. This list satisfies the desired property. Hence $H$ must be a bijection.  
\end{proof}
\begin{remark}
We can observe that in our situation only the uniqueness of compatible triple for a noncompact positive root $\gamma \notin \Pi$ can be obtained from the Restricted Root Theorem directly. 
\end{remark}

\par  From the lemma we can easily obtain that 
\begin{proposition}\label{unique}
If $C$ is a curve with generic tangent vectors in $M$, then for every point $x\in C$, $C$ is tangent to a unique standard model of diagonal curves in maximal polyspheres to the second order at $x$. 
\end{proposition}
\begin{remark}
From above lemma we can also compute the splitting type of $T(M)$ over the diagonal curve $C$.  The corresponding coroot of $\alpha_i$ is denoted by $h_{\alpha_i}$. The tangent vector of $C$ at $o$ can be written as $\sum_{i=1}^re_{\alpha_i}$. By Grothendieck's decomposition Theorem we know that $T(M)|_{C}=\bigoplus_{\gamma}\mathcal{O}(a_\gamma)$ where $a_{\gamma}=\sum_{i=1}^r\gamma(h_{\alpha_i})$ whenever $\gamma$ is a positive noncompact root.
\par It is easy to see that when $\gamma=\alpha_i$ for some $1\leq i\leq r$, $a_{\gamma}=2$. For other roots, we can actually use the proof as in the above lemma, the proof tells us for the tube type irreducible compact Hermitian symmetric space, the compatible triple defined as above exists for $\gamma$ and it is unique. Hence there exist only two roots $\alpha_k,\alpha_\ell\in \Pi$ such that $\gamma(h_{\alpha_k})=\gamma(h_{\alpha_\ell})=1$ and for other $\alpha_i\in \Pi$, $\gamma(h_{\alpha_i})=0$. Hence the splitting type must be $T(M)|_{C}=\mathcal{O}(2)^n$. Then the local uniqueness of standard model tangent to $C$ to the second order at one point is guaranteed by deformation theory.
\end{remark}

\subsection{The affine-algebraicity of the Chow space}
\par From previous construction, we obtain a holomorphic map from $C$ to the Chow space of standard models (as tangency to the second order is a holomorphic condition). We now prove that the Chow space $\mathcal{M}$ is affine algebraic and hence the holomorphic map is actually constant. We know $G$ acts transitively on $\mathcal{M}$. Denoting the stablizer of the diagonal curve by $H$ , we need the following theorem from \cite{MR0437549} and also \cite[p.162]{MR1334091}).
\begin{theorem}
Let $H$ be a closed subgroup of the complex connected reductive affine algebraic group $G$, then $G/H$ is an affine variety if and only if it is stein, if and only if $H$ is a reductive group. 
\end{theorem}
\par Write $v=\sum_{i=1}^re_{\alpha_i}$. We have the following lemma on the $K$-orbit $K.[v]\subset \mathbb{P}T_o(M)$. 
\begin{lemma}
When $M$ is an irreducible compact Hermitian symmetric space of tube type, the $K$-orbit $\mathcal{S}=K.[v]$ in the projectivized  tangent space $\mathbb{P}T_o(M)$ is an $(n-1)$-dimensional totally real submanifold, where $n=\dim(M)$.  
\end{lemma}
\begin{proof} We use $\mathfrak{l}$ to denote the Lie algebra of $K$. Then the tangent space of $\mathcal{S}$ is identified with 
\[T^{\mathbb{R}}_{[v]}(\mathcal{S})=([\mathfrak{l},v]+\mathbb{C}v)/\mathbb{C}v\]
On the other hand, we write \[\mathfrak{l}=\sum_{\phi\in \Delta}\sqrt{-1}h_{\phi}\mathbb{R}+\sum_{\phi\in \Delta_K^+}\sqrt{-1}(e_{\phi}+e_{-\phi})\mathbb{R}+\sum_{\phi\in \Delta_K^+}-(e_{\phi}-e_{-\phi})\mathbb{R}\]
where $\Delta$ denotes the set of all roots and $ \Delta_K^+$ denotes the set of all positive compact roots.
\par For the first part, we have 
\[[\sum_{\phi\in \Delta}\sqrt{-1}h_{\phi}\mathbb{R},v]=\sum_{i=1}^r\sqrt{-1}e_{\alpha_i}\mathbb{R}\]
For the second and third part, we use the following special property of the root system of Hermitian symmetric space of tube type which can be obtained from the proof of Lemma  \ref{bijection} together with the Restricted Root Theorem.
\par ($\star$) For any positive compact root $\phi$, if $\phi+\alpha_i$ is a positive noncompact root for some $\alpha_i\in \Pi$, then there exists a unique pair $\alpha_i,\alpha_j\in \Pi$ with $j\neq i$ such that $\alpha_i+\phi$ and $\alpha_j-\phi$ are noncompact positive roots. Moreover, $(\alpha_i,\alpha_j, \alpha_i+\phi)$ is the unique compatible triple for $\alpha_j-\phi$.
\par  We know for any root $\alpha,\beta$ whose sum is also a root, $[e_{\alpha},e_{\beta}]=N_{\alpha,\beta}e_{\alpha+\beta}$ for some nonzero real number $N_{\alpha,\beta}$ (cf. Theorem 5.5 in \cite[Chapter III]{MR1834454} ). Let $\beta_1=\alpha_i+\phi, \beta_2=\alpha_j-\phi$.
\par If $\beta_1\neq \beta_2$, then  $[\sqrt{-1}(e_{\phi}+e_{-\phi}),v]=\sqrt{-1}( N_{\phi,\alpha_i}e_{\beta_1}+N_{-\phi,\alpha_j}e_{\beta_2})$ and $[e_{\phi}-e_{-\phi},v]= N_{\phi,\alpha_i}e_{\beta_1}-N_{-\phi,\alpha_j}e_{\beta_2}$, the set of such compact roots $\phi$ is denoted by $\Delta_{K,1}$.
\par If $\beta_1=\beta_2$, then  $[\sqrt{-1}(e_{\phi}+e_{-\phi}),v]=\sqrt{-1}( N_{\phi,\alpha_i}+N_{-\phi,\alpha_j})e_{\beta_1}$ and $[e_{\phi}-e_{-\phi},v]= (N_{\phi,\alpha_i}-N_{-\phi,\alpha_j})e_{\beta_1}$. Note that $\{\alpha_i,\beta_1,\alpha_j\}$ is a maximal $\phi$-chain, also from Theorem 5.5 in \cite[Chapter III]{MR1834454} we have $N_{\phi,\alpha_i}^2=\phi(H_{\phi})=N_{-\phi,\alpha_j}^2$, hence either $[\sqrt{-1}(e_{\phi}+e_{-\phi}),v]=0$ or $[e_{\phi}-e_{-\phi},v]=0$, the set of compact roots $\phi$ with $[e_{\phi}-e_{-\phi},v]=0$ is denoted by $\Delta_{K,2}$ and the set of compact roots $\phi$ with $[\sqrt{-1}(e_{\phi}+e_{-\phi}),v]=0$ is denoted by $\Delta_{K,3}$.
\par Since all $N_{\phi,\alpha_i}$ are real, from above we have \begin{align*}
[\mathfrak{l},v]&=\sum_{i=1}^r\sqrt{-1}e_{\alpha_i}\mathbb{R}+\sum_{\phi\in \Delta_{K,1}}\sqrt{-1}( N_{\phi,\alpha_{i_\phi}}e_{\alpha_{i_\phi}+\phi}+N_{-\phi,\alpha_{j_\phi}}e_{\alpha_{j_\phi}-\phi})\mathbb{R}+
\\&\sum_{\phi\in \Delta_{K,1}}( N_{\phi,\alpha_{i_\phi}}e_{\alpha_{i_\phi}+\phi}-N_{-\phi,\alpha_{j_\phi}}e_{\alpha_{j_\phi}-\phi})\mathbb{R}+\sum_{\phi\in \Delta_{K,2}}\sqrt{-1}e_{\alpha_{i_\phi}+\phi}\mathbb{R}+\sum_{\phi\in\Delta_{K,3}}e_{\alpha_{i_\phi}+\phi}\mathbb{R}
\end{align*}
where $\alpha_{i_\phi},\alpha_{j_\phi}$ are uniquely determined corresponding to the choice of $\phi$.
\par Since the complex structure $J$ on $\mathbb{P}T_o(M)$ acts on $\mathfrak{m}^+$ by $\sqrt{-1}$ and also the compatible triple uniquely exists for each positive noncompact root, we can easily obtain that $[\mathfrak{l},v]^{\mathbb{C}}=\mathfrak{m}^+$ and $[\mathfrak{l},v]$ is totally real, i,e, $J[\mathfrak{l},v]\cap [\mathfrak{l},v]=\{0\}$. Therefore $\mathcal{S}$ is an $(n-1)$-dimensional totally real submanifold in $\mathbb{P}T_o(M)$. 
\end{proof}

\begin{remark}
This lemma is actually equivalent to the fact that the Bergman-\v{S}ilov boundary of a tube type irreducible bounded symmetric domain is totally real and has half real dimension of the whole space. Although the latter is a well-known result in \cite{MR0174787}, we still give a self-contained proof here.
\end{remark}

\par  Now we want to prove that $H$ is a reductive Lie group. We need to show that $H$ has a compact real form and $H$ has finitely many connected components. The latter will be uniformly explained in Proposition \ref{finite connected}. So we just show that $H$ has a compact real form in this section.

\par We consider the $G_c$ action on a fixed diagonal curve $C_d$. The stablizer is denoted by $H_0$, we claim that $Lie(H)=Lie(H_0)^{\mathbb{C}}$, since $H_0$ is a compact real Lie group, we know $H$ is reductive. 
\par $Lie(H)=Lie(H_0)^{\mathbb{C}}$ can be done by dimension counting. We consider two double fibrations. The first one is given by 
\[M=G_c/K\stackrel{\rho_1}\longleftarrow \mathcal{U}_1 \stackrel{\mu_1}\longrightarrow G_c/H_0 \]
where $\rho_1^{-1}(o)\cong K.[v]$ and $\mu_1^{-1}(\kappa)\cong \mathbb{P}^1$ for some standard model of diagonal curve $\kappa$. $G_c/H_0$ is a subspace in the Chow space of all standard models of the diagonal curve. $\mu_1$ can be well-defined for the following reason: Choose some canonical K\"{a}hler-Einstein metric on $M$ such that the diagonal curve $C_d$ is totally geodesic (so $G_c/H_0$ is actually the Chow space of all standard models of the diagonal curves which are totally geodesic with respect to this metric). $K$ is in the group of isometry $G_c$, so $K$-action on $C_d$ preserving the tangent vector $v$ will fix the standard model. From this double fibration we easily compute that 
\[
\dim_{\mathbb{R}}(H_0)=-\dim_{\mathbb{R}}(M)-\dim_{\mathbb{R}}(K.[v])+\dim_{\mathbb{R}}(G_c)+2=\dim_{\mathbb{R}}(G_c)-(3n-3)
\]
On the other hand, from Proposition \ref{unique} we have another double fibration
\[ \mathcal{D}(M) \stackrel{\rho_2}\longleftarrow \mathcal{U}_2 \stackrel{\mu_2}\longrightarrow G/H \]
where $\mathcal{D}(M)$ is a fibration over $M=G/P$ with fibre isomorphic to $P.[v]$ and $\rho_2^{-1}(o,[w])=[w,[\mathfrak{m}^-,w]]/\mathbb{C}w $ and $\mu_2^{-1}(\kappa)\cong \mathbb{P}^1$ for some standard model of diagonal curve $\kappa$. Thus we have 
\[
\dim_{\mathbb{C}}(H)=-\dim_{\mathbb{C}}(M)-\dim_{\mathbb{C}}(P.[v])-(n-1)+\dim_{\mathbb{C}}(G)+1=\dim_{\mathbb{C}}(G)-(3n-3)
\]
Combine the dimensions above, we have 
\[\dim_{\mathbb{C}}(H)=\dim_{\mathbb{R}}(H_0)\]
Since $H_0^{\mathbb{C}}\subset H$ we know $Lie(H_0)^{\mathbb{C}}=Lie(H)$ and then $H$ is a reductive group.
So we can obtain Theorem \ref{diagonal curve}.

\begin{remark}
The (complex) dimension of $G/H$ is $3n-3$, this is compatible with the computation from deformation theory, as $T(M)|_C=\mathcal{O}(2)^n$ and $\dim H^0(C,N_{C|M})=3(n-1)$.
\end{remark}

\section{The affine-algebraicity of the Chow space and a weaker version of gap rigidity}\label{weakergap}
In this part, we give the proof of Theorem \ref{generalaffineness} and Corollary \ref{general}. Suppose $X$ is an $(H_2)$-subspace in $M=G/P$ satisfying that $P.[T_o(X)]$ is contained in a complement of a $P$-invariant hypersurface in $Gr(m,T_o(M))$ where $\dim(M)=n, \dim(X)=m$. The idea is basically similar as diagonal curve case, we need to show
\begin{enumerate}
\item The standard models with the same tangent subspace at $o$ form an $(n-m)$-dimensional family and the space of second fundamental forms of all standard models is an $(n-m)$-dimensional subspace in  $S^2T_o(X)\otimes N_{o,X|M}$ (which will be denoted by $Q$). 
\item	
There is a totally real $K$-orbit on $P.[T_o(X)]\subset Gr(m, T_o(M))$ such that 
$\dim_{\mathbb{C}}(P.[V])=\dim_{\mathbb{C}}(K^{\mathbb{C}}.[V])=\dim_{\mathbb{R}}(K.[V])$ where $[V]$ denotes a reference point on the totally real $K$-orbit. Also, the isotropy group $H$ of the $G$-action on the space of standard models has finitely many connected components. 
\end{enumerate}
Assuming these two facts, we can give a proof for Theorem \ref{generalaffineness} and Corollary \ref{general}. 
\begin{proof}[Proof of Theorem \ref{generalaffineness}]
We know
$V$ is corresponding to some choice of standard models, we fix one as $g.X$ and choose some K\"{a}hler-Einstein metric on $M$ such that $g.X$ is totally geodesic with respect to this metric. Inheriting the same argument and notations in diagonal curve case, we also have two double fibrations. The first one is given by 
\[M=G_c/K\stackrel{\rho_1}\longleftarrow \mathcal{U}_1 \stackrel{\mu_1}\longrightarrow G_c/H_0 \]
where $\rho_1^{-1}(o)\cong K.[V]$ and $\mu_1^{-1}(\kappa)\cong X$ for some standard model $\kappa$, $H_0$ is the isotropy group of $G_c$-action on the Chow space of standard models with reference point $[g.X]$. The second double fibration is
\[ \mathcal{D}(M) \stackrel{\rho_2}\longleftarrow \mathcal{U}_2 \stackrel{\mu_2}\longrightarrow G/H \]
where $\mathcal{D}(M)$ is a fibration over $M=G/P$ with fibres isomorphic to $P.[V]$ and $\rho_2^{-1}(o,[V])\cong Q$ and $\mu_2^{-1}(\kappa)\cong X$ for some standard model $\kappa$, $G/H$ is the Chow space of standard models.

From the first fact 
we still have the dimension counting 
\[	\dim_{\mathbb{R}}(H_0)=-\dim_{\mathbb{R}}(M)-\dim_{\mathbb{R}}(K.[V])+\dim_{\mathbb{R}}(G_c)+2m
\]
and
\[
\dim_{\mathbb{C}}(H)=-\dim_{\mathbb{C}}(M)-\dim_{\mathbb{C}}(P.[V])-(n-m)+\dim_{\mathbb{C}}(G)+m
\]
Therefore $\dim_{\mathbb{R}}(H_0)=\dim_{\mathbb{C}}(H)$ holds, together with the fact that $H$ has finitely many connected components, we know $H$ is reductive and the conclusion easily follows. 
\end{proof}

Then we can prove Corollary \ref{general}. 
\begin{proof}[Proof of Corollary \ref{general}]
From the assumption and the first fact we know that for each $x\in S$ there exists only one standard model tangent to $S$ at $x$ to the second order and hence the holomorphic map from $S$ to the Chow space of standard models can be constructed. Then the conclusion follows from Theorem \ref{generalaffineness}.
\end{proof}
\begin{remark}
For the curve case, i.e. $M$ is irreducible and of tube type, $X$ is a diagonal curve in $M$, two conditions are satisfied including $P.[T_o(X)]=\mathcal{O}_o$ and  $Q=S^2T_o(X)\otimes N_{o,X|M}$ by dimension counting, the second condition gives the existence of standard models with second order tangency automatically. Hence the 'weaker' gap rigidity is actually the original gap rigidity in this case. In general situation we want to remove the extra condition on second order tangency, i.e. to construct a holomorphic map from $S$ to the Chow space of standard models with only the first order tangency. This may need further study.
\end{remark}

We firstly give an affirmative answer to the second statement, we have 
\begin{lemma}\label{totally real}
If the $P$-orbit $P.[T_o(X)]\subset Gr(m, T_o(M))$ is contained in a complement of a $P$-invariant hypersurface, then there is a totally real $K$-orbit in $P.[T_o(X)]\subset Gr(m, T_o(M))$, denoting the reference point by $[V]$ we have  \[\dim_{\mathbb{C}}(K^{\mathbb{C}}.[V])=\dim_{\mathbb{R}}(K.[V])\]
\end{lemma}
\begin{proof}
Since $K^{\mathbb{C}}.[V]$ is in the complement of a hypersurface in $Gr(p,T_o(M))$, it is affine algebraic and hence the isotropy group $J$ of this orbit is reductive. Since $K$ is a compact real subgroup of $K^{\mathbb{C}}$, there is a totally real $K$-orbit on $K^{\mathbb{C}}/J$ (cf.\cite[p.161]{MR1334091}). We let $L=K\cap J$, then
\[\dim_{\mathbb{R}}K/L\leq \dim_{\mathbb{C}}K^{\mathbb{C}}/J\]
so $\dim_{\mathbb{R}}(L)\geq \dim_{\mathbb{C}}(J)$. On the other hand we know the Lie algebra $Lie(L)\oplus \sqrt{-1}Lie(L) \subset J$, then $L$ is a maximal compact real form of $J$ and \[\dim_{\mathbb{R}}K/L= \dim_{\mathbb{C}}K^{\mathbb{C}}/J\]
\end{proof}
\begin{remark}
This lemma gives a uniform conceptual proof on the existence of totally real $K$-orbits, and the computation using root vectors in curve case gives an explicit example.
\end{remark}

Also we can prove 
\begin{proposition}\label{finite connected}
The isotropy group $H$ has finitely many connected components.
\end{proposition}
\begin{proof}
We consider the double fibration
\[ \mathcal{D}(M) \stackrel{\rho_2}\longleftarrow \mathcal{U}_2 \stackrel{\mu_2}\longrightarrow G/H \]
We know $\mathcal{D}(M)$ is a fibration over $M=G/P$ with fibre isomorphic to $P.[V]=K^{\mathbb{C}}.[V]$ and the fibre of $\rho_2$ is some affine space. On the other hand the fibre of $\mu_2$ is some standard model of $(H_2)$-subspace. We will use some exact sequences on the fundamental groups in the fibrations. The first exact sequence is 

\[   \pi_1(K^{\mathbb{C}}.[V])\longrightarrow \pi_1(\mathcal{D}(M)) \longrightarrow \pi_1(M)  \]
Denoting the connected semisimple part of $K^{\mathbb{C}}$ by $\tilde{K}^{\mathbb{C}} $, we know $K^{\mathbb{C}}.[V]=\tilde{K}^{\mathbb{C}}.[V]$ is stein and the isotropy group $J$ under $\tilde{K}^{\mathbb{C}}$-action is reductive and hence has finitely many connected components. Then from  $\pi_{1}(\tilde{K}^{\mathbb{C}})\longrightarrow\pi_1(K^{\mathbb{C}}.[V])\longrightarrow \pi_0(J)\longrightarrow 0$ we know $\pi_1(K^{\mathbb{C}}.[V])$ is finite (Although $\pi_0$ is not a group, the sequence is still exact in the sense that kernel equals image, and the 'zero' in $\pi_{0}(J)$ can be chosen as any element in $\pi_0(J)$). On the other hand $\pi_1(M)=0$, so $\pi_1(\mathcal{D}(M))$ is finite. 
\par The second exact sequence is \[    0 \longrightarrow \pi_1(\mathcal{U}_2) \longrightarrow \pi_1(\mathcal{D}(M)) \]
since the fibre of $\rho_2$ is an affine space. Thus $\pi_1(\mathcal{U}_2)$ is finite.
\par The last exact sequence is 
\[    0 \longrightarrow \pi_1(\mathcal{U}_2) \longrightarrow \pi_1(G/H) \longrightarrow 0\]
since any $(H_2)$-subspace is path-connected and  simply-connected. Hence $\pi_1(G/H)$ is finite and from 
$  \pi_1(G) \longrightarrow \pi_1(G/H) \longrightarrow \pi_0(H) \longrightarrow 0$
we know $H$ has finitely many connected components.
\end{proof}

\par  The remaining task is to show the first statement, we still have the following lemma as diagonal curve case. Hereafter all the complex conjugations are with respect to the compact real form.
\begin{lemma}\label{2nd fundamental form under M-}
Suppose that $X\subset M$ is totally geodesic with respect to some choice of K\"{a}hler-Einstein metric, then under the $M^-$ action given by $\exp(\overline{u})$ for some $u\in \mathfrak{m}^+$, the second fundamental form of $\exp(\overline{u}).X$ in $M$ will be \[\sigma_{o}(v_1,v_2)=[v_1,[\overline{u},v_2]] \mod T_{o}(X)\] for any $v_1,v_2\in T_o(X)$.
\end{lemma}
\begin{proof}
For any $v\in T_o(X)$, from \cite[Lemma 4.3]{MR1198602} we know $\sigma_{o}(v,v)=[v,[\overline{u},v]] \mod T_{o}(X)$. By polarization argument the conclusion follows. 
\end{proof}

Then we can prove the following lemma which reduces the proof of the first statement to check a Lie bracket generating condition, we have 
\begin{lemma}
If $\text{span}[T_o(X),[\mathfrak{m}^-,T_o(X)]]=\mathfrak{m^+}=T_o(M)$, then for $u\in \mathfrak{m}^+$, $\exp{\overline{u}}.X=X$ if and only if $u\in T_o(X)$, and then $\exp(\mathfrak{m}^-).X$ gives an $(n-m)$-dimensional family of standard models.
\end{lemma}
\begin{proof}
Since $X$ is a geodesic model, we know $[T_o(X),[\overline{T}_o(X),T_o(X)]] \subset T_o(X)$ and hence if $u\in T_o(X)$, $\exp{\overline{u}}.X=X$. For another direction, if $\exp{\overline{u}}.X=X$, then $[v,[\overline{u},v]]\in T_o(X)$ for all $v\in T_o(X)$, under the metric induced from the  Killing form $B$ we decompose $\mathfrak{m}^+=T_o(X)\oplus T_o(X)^{\bot}$. Then we have $B([v,[\overline{u'},v]], \overline{w})=0$ for all $w\in T_o(X)^\bot$ if we denote the projection of $u$ to $T_o(X)^{\bot}$ by $u'$. Hence $B([v,[\overline{w},v]], \overline{u'})=0$ for all $w\in T_o(X)^\bot$, since $\text{span}[T_o(X),[\mathfrak{m}^-,T_o(X)]]=\mathfrak{m^+}$ and $[T_o(X),[\overline{T}_o(X),T_o(X)]] \subset T_o(X)$ we know $u'$ is orthogonal to the whole $\mathfrak{m}^+$ thus $u'=0$ and $u\in T_o(X)$.
\end{proof}

Therefore it suffices to show the following Lie bracket generating condition 
\[(\dag):\text{span}[T_o(X),[\mathfrak{m}^-,T_o(X)]]=\mathfrak{m}^+\]
This condition is invariant under $P$-action, so without loss of generality we may assume that  $V=T_o(X)$.
We will check that all $(H_2)$-subspaces in irreducible compact Hermitian symmetric spaces with rank $\geq 2$ satisfy this condition. Classification of maximal $(H_2)$-subspace is given in the table in \cite[pp.27-28]{MR2127948}, an easy corollary obtained from the diagonal curve case shows that if 
$V$ contains a tangent vector of maximal rank on $M$ and $M$ is of tube type, then the Lie bracket generating condition $(\dag)$ will hold. We have typical examples of higher dimensional submanifolds for this corollary, like smooth linear section subquadrics in a hyperquadric and maximal polyspheres in tube type Hermitian symmetric spaces. Motivated by this, if $T_o(X)$ contains such a tangent vector, we will call the embedding $X\subset M$ is of diagonal type and the others are called non-diagonal type. We want to see whether the maximal $(H_2)$-subspaces will be of diagonal type. In the meantime, we will divide the cases into two categories from the classification, when the ambient space is of Type I,II,III, we will prove the generating condition based on the matrices expressions; when the ambient space is of Type IV or Type $E_6$, $E_7$, all $(H_2)$-subspaces can be easily written down and the checking is straightforward.

\subsection {When $M$ is a hyperquadric}
All $(H_2)$-subspaces in a hyperquadric can be classified as linear section subquadrics and diagonal curves in geodesic two-spheres. All of these are of diagonal type and hyperquadrics are of tube type, so the Lie bracket generating condition $(\dag)$ has already been satisfied.
\subsection{When $M$ is of $E_6$-type}
From the classification in \cite{MR2127948} we know all $(H_2)$-embedding in $E_6$ type ambient space is of diagonal type (but $E_6$ is not of tube type), and root systems will be used in this case, we have
\begin{lemma}\label{E6 span}
When $M$ is the $E_6$-type irreducible compact Hermitian symmetric space, for any $(H_2)$-embedding of $X\subset M$, denoting the tangent space of $X$ at the reference point $o$ by $V$, then $[V,[\mathfrak{m}^-,V]]$ spans the whole tangent space $T_o(M)\cong \mathfrak{m}^+$.
\end{lemma}
\begin{proof}
We refer the reader to the table given in \cite[p.28]{MR2127948}, chains of $(H_2)$-subspaces can be classified as $\mathbb{P}^2 \stackrel{diag}{\hookrightarrow}\mathbb{P}^2\times \mathbb{P}^2 \hookrightarrow G(2,4) \hookrightarrow M$ or $\mathbb{P}^5\times \mathbb{P}^1\hookrightarrow M$. It suffices to check $\mathbb{P}^2$ and $\mathbb{P}^5\times \mathbb{P}^1\hookrightarrow M$ cases. We adapt the root system notations in \cite[p.290]{MR0214807}, where the extended Dynkin diagram is as follows

%\begin{figure}
\begin{figure}[h]
	\begin{tikzpicture}
	\draw[thick] (0,0) -- (6,0)  (3,0) -- (3,-1.5) (3,-1.5) -- (1.5,-1.5);
	
	\draw[ thick, fill=black] (0,0) circle (3pt) node[above, outer sep=3pt]{$\alpha_1$};
	\draw[ thick, fill=white] (1.5,0) circle (3pt) node[above, outer sep=3pt]{$\alpha_2$};
	\draw[ thick, fill=white] (3,0) circle (3pt) node[above, outer sep=3pt]{$\alpha_3$};
	\draw[ thick, fill=white] (4.5,0) circle (3pt) node[above,outer sep=3pt]{$\alpha_4$};
	\draw[ thick, fill=white] (6,0) circle (3pt) node[above,outer sep=3pt]{$\alpha_5$};
	\draw[ thick, fill=white] (3,-1.5) circle (3pt) node[right,outer sep=3pt]{$\alpha_6$};
	\draw[ thick, fill=black] (1.5,-1.5) circle (3pt) node[above,outer sep=3pt]{$-\gamma$};
	\end{tikzpicture} 
	\caption{Extended Dyndin diagram of $E_6$} \label{fig:E6}
\end{figure}
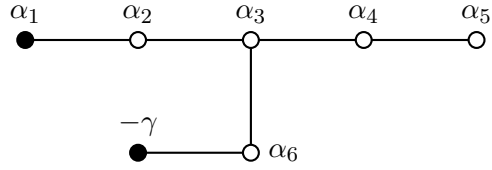	%\end{figure}

Here the highest root $\gamma=\alpha_1+2\alpha_2+3\alpha_3+2\alpha_4+\alpha_5+2\alpha_6$. Then the tangent space of $X=\mathbb{P}^2$ is spanned by $e_{\alpha_1}+e_{\gamma-\alpha_6}$ and $e_{\alpha_1+\alpha_2}+e_{\gamma}$. All the positive noncompact roots can be listed according to the coefficients of $\alpha_1,...,\alpha_6$ in order, which are \[
\begin{gathered}
(100000)=\alpha_1, (110000)=\alpha_1+\alpha_2, (111000), (111001) \\
(111100), (111101), (111110), (112101) \\
(111111), (122101), (112111), (122111) \\
(112211), (122211), (123211)=\gamma-\alpha_6, (123212)=\gamma \\
\end{gathered}\]
One can easily check that any positive noncompact root $\delta$ has a pairing $(\delta,\delta')$ such that $\delta+\delta'=\alpha_1+\gamma-\alpha_6$ or $\delta+\delta'=\alpha_1+\alpha_2+\gamma$, thus every positive noncompact root vector is in the space span$[V,[\mathfrak{m}^-,V]]$.

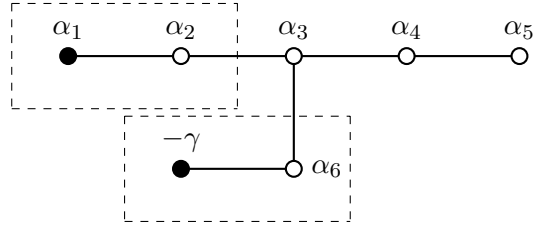
\begin{figure}[h]
	\begin{tikzpicture}
	\draw[thick] (0,0) -- (6,0)  (3,0) -- (3,-1.5) (3,-1.5) -- (1.5,-1.5);
	\draw[dashed](-0.75,0.7)--(2.25,0.7)
	(-0.75,-0.7)--(2.25,-0.7) (-0.75,0.7)--(-0.75,-0.7) (2.25,0.7)--(2.25,-0.7);
	\draw[dashed](0.75,-0.8)--(3.75,-0.8)
	(0.75,-2.2)--(3.75,-2.2) (0.75,-0.8)--(0.75,-2.2) (3.75,-0.8)--(3.75,-2.2);
	
	\draw[ thick, fill=black] (0,0) circle (3pt) node[above, outer sep=3pt]{$\alpha_1$};
	\draw[ thick, fill=white] (1.5,0) circle (3pt) node[above, outer sep=3pt]{$\alpha_2$};
	\draw[ thick, fill=white] (3,0) circle (3pt) node[above, outer sep=3pt]{$\alpha_3$};
	\draw[ thick, fill=white] (4.5,0) circle (3pt) node[above,outer sep=3pt]{$\alpha_4$};
	\draw[ thick, fill=white] (6,0) circle (3pt) node[above,outer sep=3pt]{$\alpha_5$};
	\draw[ thick, fill=white] (3,-1.5) circle (3pt) node[right,outer sep=3pt]{$\alpha_6$};
	\draw[ thick, fill=black] (1.5,-1.5) circle (3pt) node[above,outer sep=3pt]{$-\gamma$};
	\end{tikzpicture} 
	\caption{$(H_2)$-embedding of $\mathbb{P}^2\times \mathbb{P}^2$} \label{fig:P2}
\end{figure}

\par For the embedding $X=\mathbb{P}^5\times \mathbb{P}^1\hookrightarrow M$ we know the tangent space of $\mathbb{P}^5\times \mathbb{P}^1$
is spanned by the root vectors with roots containing no $\alpha_6$ term and the root vector $e_{\gamma}$, we can check the conclusion in a similar manner.  

\begin{figure}[h]
	\begin{tikzpicture}
	\draw[thick] (0,0) -- (6,0)  (3,0) -- (3,-1.5) (3,-1.5) -- (1.5,-1.5);
	\draw[dashed](-0.75,0.7)--(6.75,0.7)
	(-0.75,-0.7)--(6.75,-0.7) (-0.75,0.7)--(-0.75,-0.7) (6.75,0.7)--(6.75,-0.7);
	\draw[dashed](0.75,-0.8)--(2.25,-0.8)
	(0.75,-2.2)--(2.25,-2.2) (0.75,-0.8)--(0.75,-2.2) (2.25,-0.8)--(2.25,-2.2);
	
	\draw[ thick, fill=black] (0,0) circle (3pt) node[above, outer sep=3pt]{$\alpha_1$};
	\draw[ thick, fill=white] (1.5,0) circle (3pt) node[above, outer sep=3pt]{$\alpha_2$};
	\draw[ thick, fill=white] (3,0) circle (3pt) node[above, outer sep=3pt]{$\alpha_3$};
	\draw[ thick, fill=white] (4.5,0) circle (3pt) node[above,outer sep=3pt]{$\alpha_4$};
	\draw[ thick, fill=white] (6,0) circle (3pt) node[above,outer sep=3pt]{$\alpha_5$};
	\draw[ thick, fill=white] (3,-1.5) circle (3pt) node[right,outer sep=3pt]{$\alpha_6$};
	\draw[ thick, fill=black] (1.5,-1.5) circle (3pt) node[above,outer sep=3pt]{$-\gamma$};
	\end{tikzpicture} 
	\caption{$(H_2)$-embedding of $\mathbb{P}^5\times \mathbb{P}^1$} \label{fig:P5}
\end{figure}
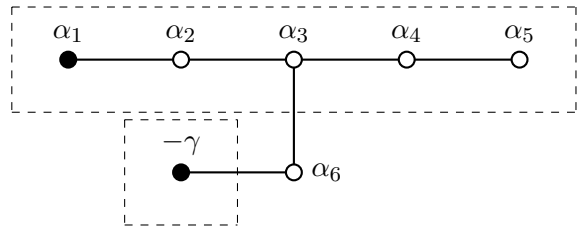

\end{proof}

\subsection{When $M$ is of $E_7$-type.} All $(H_2)$-subspaces have been classified in \cite{MR2127948}. As $E_7$ is of tube type, the diagonal type cases has been done from previous discussion, we only need to check the non-diagonal type cases, we have 
\begin{lemma}\label{E7 span}
When $M$ is the $E_7$-type irreducible compact Hermitian symmetric space, for any $(H_2)$-embedding of $X\subset M$, denoting the tangent space of $X$ at the reference point $o$ by $V$, then $[V,[\mathfrak{m}^-,V]]$ spans the whole tangent space $T_o(M)\cong \mathfrak{m}^+$.
\end{lemma}
\begin{proof}
We refer the reader to the table given in \cite[p.28]{MR2127948}, we only need to consider the non-diagonal type, chains of such $(H_2)$-subspaces can be classified as $\mathbb{P}^3 \stackrel{diag}{\hookrightarrow}\mathbb{P}^3\times \mathbb{P}^3 \hookrightarrow G(2,6) \hookrightarrow M$ or $\mathbb{P}^5\times \mathbb{P}^2\hookrightarrow M$. It suffices to check $\mathbb{P}^3$ and  $\mathbb{P}^5\times \mathbb{P}^2\hookrightarrow M$ cases. We adapt the root system notations in \cite[p.291]{MR0214807}, the extended Dynkin diagram is as follows

\begin{figure}[h]
	\begin{tikzpicture}
	\draw[thick] (0,0) -- (9,0)  (4.5,0) -- (4.5,-1.5) ;
	
	\draw[ thick, fill=black] (0,0) circle (3pt) node[above, outer sep=3pt]{$\alpha_1$};
	\draw[ thick, fill=white] (1.5,0) circle (3pt) node[above, outer sep=3pt]{$\alpha_2$};
	\draw[ thick, fill=white] (3,0) circle (3pt) node[above, outer sep=3pt]{$\alpha_3$};
	\draw[ thick, fill=white] (4.5,0) circle (3pt) node[above,outer sep=3pt]{$\alpha_4$};
	\draw[ thick, fill=white] (6,0) circle (3pt) node[above,outer sep=3pt]{$\alpha_5$};
	\draw[ thick, fill=white] (7.5,0) circle (3pt) node[above,outer sep=3pt]{$\alpha_6$};
	\draw[ thick, fill=black] (9,0) circle (3pt) node[above,outer sep=3pt]{$-\gamma$};
	\draw[ thick, fill=white] (4.5,-1.5) circle (3pt) node[right,outer sep=3pt]{$\alpha_7$};
	\end{tikzpicture}
	\caption{Extended Dyndin diagram of $E_7$}
	\label{fig:E7}
\end{figure}
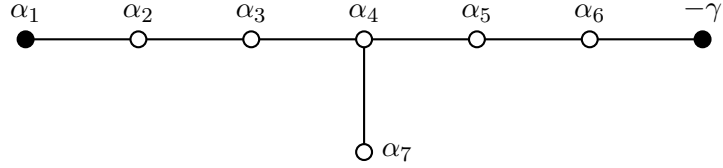
Here the highest root $\gamma=\alpha_1+2\alpha_2+3\alpha_3+4\alpha_4+3\alpha_5+2\alpha_6+2\alpha_7$. Then the tangent space of $X=\mathbb{P}^3$ is spanned by $e_{\alpha_1}+e_{\gamma}$, $e_{\alpha_1+\alpha_2}+e_{\gamma-\alpha_6}$ and $e_{\alpha_1+\alpha_2+\alpha_3}+e_{\gamma-\alpha_5-\alpha_6}$. All the positive noncompact roots can be listed according to the coefficients of $\alpha_1,...,\alpha_7$ in order, which are \[
\begin{gathered}
(1000000)=\alpha_1, (1100000)=\alpha_1+\alpha_2, (1110000)=\alpha_1+\alpha_2+\alpha_3\\ 
(1111000), (1111001), (1111100), (1111110) \\
(1111101), (1111111), (1112101), (1112111) \\
(1122101), (1122111), (1122111), (1222101) \\
(1122211), (1222111),
(1123211), (1222211)\\
(1123212), (1223211),
(1223212), (1233211),  (1233212),
\\
(1234212)=\gamma-\alpha_6-\alpha_5,
(1234312)=\gamma-\alpha_6, (1234322)=\gamma \\
\end{gathered}\]
One can easily check that any positive noncompact root $\delta$ has a pairing $(\delta,\delta')$ such that $\delta+\delta'=\alpha_1+\gamma$ or $\delta+\delta'=\alpha_1+\alpha_2+\gamma-\alpha_6$ or $\delta+\delta'=\alpha_1+\alpha_2+\alpha_3+\gamma-\alpha_5-\alpha_6$, thus every positive noncompact root vector is in the space span$[V,[\mathfrak{m}^-,V]]$.

\begin{figure}[h]
	\begin{tikzpicture}
	\draw[thick] (0,0) -- (9,0)  (4.5,0) -- (4.5,-1.5) ;
	\draw[dashed](-0.75,0.7)--(3.75,0.7)
	(-0.75,-0.7)--(3.75,-0.7) (-0.75,0.7)--(-0.75,-0.7) (3.75,0.7)--(3.75,-0.7);
	\draw[dashed](5.25,0.7)--(9.75,0.7)
	(5.25,-0.7)--(9.75,-0.7) (5.25,0.7)--(5.25,-0.7) (9.75,0.7)--(9.75,-0.7);
	
	\draw[ thick, fill=black] (0,0) circle (3pt) node[above, outer sep=3pt]{$\alpha_1$};
	\draw[ thick, fill=white] (1.5,0) circle (3pt) node[above, outer sep=3pt]{$\alpha_2$};
	\draw[ thick, fill=white] (3,0) circle (3pt) node[above, outer sep=3pt]{$\alpha_3$};
	\draw[ thick, fill=white] (4.5,0) circle (3pt) node[above,outer sep=3pt]{$\alpha_4$};
	\draw[ thick, fill=white] (6,0) circle (3pt) node[above,outer sep=3pt]{$\alpha_5$};
	\draw[ thick, fill=white] (7.5,0) circle (3pt) node[above,outer sep=3pt]{$\alpha_6$};
	\draw[ thick, fill=black] (9,0) circle (3pt) node[above,outer sep=3pt]{$-\gamma$};
	\draw[ thick, fill=white] (4.5,-1.5) circle (3pt) node[right,outer sep=3pt]{$\alpha_7$};
	\end{tikzpicture}
	\caption{$(H_2)$-embedding of $\mathbb{P}^3\times \mathbb{P}^3$}
	\label{fig:P3}
\end{figure}
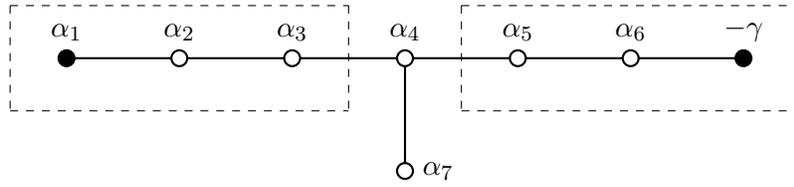

\par For the embedding $X=\mathbb{P}^5\times \mathbb{P}^2\hookrightarrow M$ we know the tangent space of $\mathbb{P}^5\times \mathbb{P}^2$
is spanned by the root vectors with roots containing no $\alpha_5$ term and the root vector $e_{\gamma}$, similar procedure can be done.
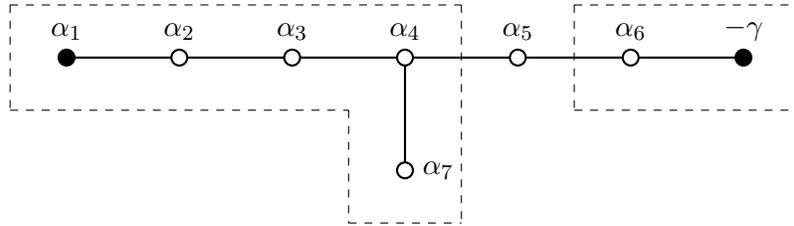
\begin{figure}[h]
	\begin{tikzpicture}
	\draw[thick] (0,0) -- (9,0)  (4.5,0) -- (4.5,-1.5) ;
	\draw[dashed](-0.75,0.7)--(5.25,0.7)
	(-0.75,-0.7)--(3.75,-0.7) (-0.75,0.7)--(-0.75,-0.7) (5.25,0.7)--(5.25,-2.2)
	(3.75,-0.7)--(3.75,-2.2) (3.75,-2.2)--(5.25,-2.2)
	;
	\draw[dashed](6.75,0.7)--(9.75,0.7)
	(6.75,-0.7)--(9.75,-0.7) (6.75,0.7)--(6.75,-0.7) (9.75,0.7)--(9.75,-0.7);
	
	\draw[ thick, fill=black] (0,0) circle (3pt) node[above, outer sep=3pt]{$\alpha_1$};
	\draw[ thick, fill=white] (1.5,0) circle (3pt) node[above, outer sep=3pt]{$\alpha_2$};
	\draw[ thick, fill=white] (3,0) circle (3pt) node[above, outer sep=3pt]{$\alpha_3$};
	\draw[ thick, fill=white] (4.5,0) circle (3pt) node[above,outer sep=3pt]{$\alpha_4$};
	\draw[ thick, fill=white] (6,0) circle (3pt) node[above,outer sep=3pt]{$\alpha_5$};
	\draw[ thick, fill=white] (7.5,0) circle (3pt) node[above,outer sep=3pt]{$\alpha_6$};
	\draw[ thick, fill=black] (9,0) circle (3pt) node[above,outer sep=3pt]{$-\gamma$};
	\draw[ thick, fill=white] (4.5,-1.5) circle (3pt) node[right,outer sep=3pt]{$\alpha_7$};
	\end{tikzpicture}
	\caption{$(H_2)$-embedding of $\mathbb{P}^5\times \mathbb{P}^2$}
	\label{fig:P5P2}
\end{figure}	
\end{proof}

\subsection{When $M$ is of Type I, II, III}
The rest are dealt by matrices expressions.
We will check whether some of the maximal $(H_2)$-subspaces are of diagonal type and then prove the Lie bracket generating condition $(\dag)$. When $X\subset M$ is of non-diagonal type or when the ambient space is not of tube type, the corrollary obtained from diagonal curve case no long works, however, we can obtain some idea from the tube type case, for instance when $M$ is a Grassmannian $G(n,n) (n\geq 2)$, the tangent space can be identified with  $n\times n$ matrices, we know a tangent vector $v$ of maximal rank correponding to a matrix of full rank generates all $n\times n$ matrices under the Lie bracket $[v,[\mathfrak{m}^-,v]]$, then in general for $G(p,q) (p\geq q\geq 2)$, we identify the tangent space as
\[ T_o(M)\cong \mathfrak{m}^+=\{\begin{bmatrix}
0 & A \\
0 & 0\\
\end{bmatrix}:A\ \text{is a $p\times q$ matrix}\}   \]
if a tangent vector $v$ is identified with a  $p\times q$ matrix with only one maximal principal submatrix of full rank, for instance $A$ is of the form
\[ \left[
\begin{BMAT}(e)[2pt,3cm,2cm]{c.c.c}{c.c}
0 & A' & 0\\ 0 & 0 &0
\end{BMAT}
\right] \]
where $A'$ is an $r\times r$ matrix of full rank. Then the Lie bracket $[v,[\mathfrak{m}^-,v]]$ will generate all matrices of the form 
\[ \left[
\begin{BMAT}(e)[2pt,3cm,2cm]{c.c.c}{c.c}
0 & * & 0\\ 0 & 0 &0
\end{BMAT}
\right] \]
If such matrices corresponding to tangent vectors in $V=T_o(X)$ can cover all the $pq$ positions in the matrix, we call the $(H_2)$-embedding satisfies the \textbf{Condition C}. One can easily obtain that if \textbf{Condition C} is satisfied, then $[V,[\mathfrak{m}^-,V]]$ will generate $\mathfrak{m}^+\cong T_o(M)$. When the ambient space is Type II or Type III, in similar mannar we just replace the matrices $A$ and the principal submatrices by antisymmetric matrices and symmetric matrices respectively. Then we reduce the problem to show that all $(H_2)$-embeddings in Type I,II,III ambient spaces satisfy \textbf{Condition C}. At first we check whether some maximal $(H_2)$-subspaces in the Grassmannian are of diagonal type, and furthermore check whether they will satisfy \textbf{Condition C}.
\subsubsection{Even dimensional hyperquadric into Grassmannian}
In this and next part, we use \cite{MR1031992} and \cite{MR0196134} as a general reference for spin representations and the embedding from hyperquadric into Grassmannian.
We check that image of the embedding $Q^{2\ell} \rightarrow G(2^{\ell-1},2^{\ell-1})$ induced by spin representation has tangent vector of maximal rank in 
$ G(2^{\ell-1},2^{\ell-1})$, i.e. the embedding is of diagonal type. Let $\mathfrak{g}_1=\mathfrak{so}(2\ell+2,\mathbb{C})$ be the Lie algebra of the automorphism group of $Q^{2\ell}$ and in block form 
\[\mathfrak{g}_1=\{\begin{bmatrix}
A & B \\
C & D\\
\end{bmatrix}:A=-A^t,D=-D^t,B=-C^t\}\]
The holomorphic tangent space can be identified with 
\[\mathfrak{m}_1^+=\{\begin{bmatrix}
0 & B \\
-B^t & 0\\
\end{bmatrix}:B=(\sqrt{-1}Z,Z), Z\ \text{is a $2\ell\times 1$ matrix}\}\]
On the other hand we consider the Clifford algebra $C=C^{\mathbb{C}}(2\ell+2)$ associated to a complex vector space $E=\mathbb{C}^{2\ell+2}$ and the standard nondegenerate quadratic form $q$ such that $q(v,v)=-v^2$ for any $v\in E$, choose a basis $\{e_1,...e_{2\ell+2}\}$ with $q(e_i,e_j)=\delta_{ij}$ we know as a vector space
\[\mathfrak{spin}(2\ell+2)=\text{span}\{e_ie_j\in C:1\leq i<j\leq 2\ell+2\}\]
and together with the operation of taking commutators it is isomorphic to $\mathfrak{so}(2\ell+2,\mathbb{C})$, and more precisely the isomorphism is given by \[ L_{ij}=E_{ij}-E_{ji}\longleftrightarrow \frac{1}{2}e_ie_j\]
where $L_{ij}$ gives a basis for antisymmetric complex $(2\ell+2)\times (2\ell+2)$ matrices.

\par We also introduce an explicit spinor module using the exterior algebra $S=\bigwedge^*W$ where $W$ is an $(\ell+1)$-dimensional complex vector space with an orthonormal basis $\{w_1,...w_{\ell+1}\}$ (for some Hermitian inner product on $W$). For any $w\in W$, it induces two operations on $S$, the first one is given by (left) exterior multiplication, we denote it by $w\wedge$, the second one is the adjoint operation of $w\wedge$ on $S$ with respect to the Hermitian inner product induced by the Hermitian inner product on $W$, we will denote it by $w\lfloor$. Then $w_i\wedge,w_i\lfloor$ generates the algebra $\text{End}(S)$ which also carries Clifford algebra structure, for simplicity we use the notation $a_i^\dagger=w_i\wedge, a_i=w_i\lfloor$. 
Let $\{ , \}$ be a nondegenerate quadratic form on the complex vector space spanned by all $a_i^{\dagger},a_i$ satisfying $\{a_i,a_j\}=\{a_i^\dagger,a_j^\dagger\}=0$
and $\{a_i,a^{\dagger}_j\}=\delta_{ij}$. The restriction $ab+ba=-\{a,b\}$ together with the nondegenerate quadratic form $\frac{1}{2}\{,\}$ gives a Clifford algebra structure of $\text{End}(S)$.

\par In fact, $W$ can be constructed by letting $w_i=\frac{1}{2}(e_{2i-1}+\sqrt{-1}e_{2i})$, where the Hermitian inner product can be chosen such that $\{w_i\}_i$ is orthonormal. In this setting we know in the Clifford algebra $C$\[
\begin{gathered}
w^2_i=0,w_iw_j=-w_jw_i\\
\overline{w}_iw_j=-w_{j}\overline{w}_i-\delta_{ij}
\end{gathered}
\]
Let $\overline{w}_{N}=\overline{w}_1\cdots\overline{w}_{\ell+1}$, then by canonically identifying $S$ with a subalgebra of $C$ generated by $W$ (still denoted by $S$), we know $S\overline{w}_N$ is a left ideal of the Clifford algebra and for each $e\in C$, there is a unique linear transformation $L(e)$ on $S$ such that $ex\overline{w}_N=L(e)x\overline{w}_N$, then through $L$ we can identify $C\cong \text{End}(S)$ as an isomorphism of Clifford algebras, this gives the spin representation by restricting the Clifford algebra on $\mathfrak{spin}(2\ell+2)$. More precisely one can easily check that the action $L(w_i)$ and $L(\overline{w}_i)$ are actually the action $a^\dagger_i$ and $-a_i$ on the exterior algebra $S$.  

\par Therefore the generators of the Clifford algebra $C$ can be identified with generators of $\text{End}(S)$ by 
\[e_{2j-1}\longleftrightarrow a_j^\dagger-a_j,  e_{2j}\longleftrightarrow -\sqrt{-1}(a_j^\dagger+a_j)\]
In this case the spin representation is not irreducible, we denote by $C^{\pm}$ the subspaces of $C$ generated by even and odd elements respectively and $C^+$ is a subalgebra. For $e\in \mathfrak{spin}(2\ell+2)\subset C^{+}$, the action $L(e)$ preserves $S^+=\bigwedge^{even}W$ and $S^-=\bigwedge^{odd}W$ and induces two so-called half-spin representations, we consider the representation on $S^+$ whose dimension is $2^\ell$ then this gives an $(H_2)$-embedding of $Q^{\ell}$ into $G(2^{\ell-1},2^{\ell-1})$.
\par Now we choose a vector in $\mathfrak{m}_1^+$ given by $Z=(-\sqrt{-1},0,...,0)^t$ then it can be identified with $\frac{1}{2}(e_1e_{2\ell+1}-\sqrt{-1}e_1e_{2\ell+2})$.
Through the spin representation $\frac{1}{2}(e_1e_{2\ell+1}-\sqrt{-1}e_1e_{2\ell+2})=-(a_1^\dagger-a_1)a_{\ell+1}$, we choose the standard basis of $S^+$ given by $w_{i_1}\wedge...\wedge w_{i_k}$ ($k$ is even), then the operation on $S^+$ annihilates the terms containing no $w_{\ell+1}$ term and this operation on the subspace generated by the terms containing $w_{\ell+1}$ has $2^{\ell-1}$-dimensional image. Thus the image of the tangent vector given by $Z=(-\sqrt{-1},0,...,0)^t$ is a tangent vector of maximal rank in the tangent space of $G(2^{\ell-1},2^{\ell-1})$.

\subsubsection{Odd dimensional hyperquadric into Grassmannian}
Next we continue to check that the embedding $Q^{2\ell-1} \rightarrow G(2^{\ell-1},2^{\ell-1})$ induced by spin representation is also of diagonal type. In this case the representation is irreducible. Let $C=C^{\mathbb{C}}(2\ell+1)$ be the Clifford algebra associated to a complex vector space $E=\mathbb{C}^{2\ell+1}$ and the standard nondegenerate quadratic form $q$, we still denote the basis by $e_1,...,e_{2\ell+1}$ satisfying $q(e_i,e_j)=\delta_{ij}$. Let $E'$ be the subspace generated by $e_i(i
\leq 2\ell-2), e_{2\ell}, e_{2\ell+1}$, there is a Clifford algebra $C'$ associated to $E'$ together with the standard nondegenerate quadratic form induced from $q$. Since $q(x',x')=-x'^2=-(x'e_{2\ell-1})^2$, we know $(C',q)$ is isomorphic to $(C^+,q)$ as a Clifford algebra induced by $\psi(x')=x'e_{2\ell-1}$ for $x'\in E'$. Now the action of $e_{2\ell-1}$ on the spin module is identity. Similar procedure as even dimensional hyperquadric case can be done for $C'$, then we also obtain that the image of the tangent vector given by $Z=(-\sqrt{-1},0,...,0)$ is a tangent vector of maximal rank in the tangent space of $G(2^{\ell-1},2^{\ell-1})$.

\subsubsection{Project space into Grassmannian}
We will check that the embedding of $\mathbb{P}^n (n\geq 3)$ into $G(p,q)$ through the irreducible representation $\Lambda_m$ is not of diagonal type, where $\Lambda_m$ is the space of skew symmetric tensors of degree $m$ ($1< m< n$), and $q=\binom{n}{m}, p=\binom{n}{m-1}$. However the Lie bracket generating condition $(\dag)$ is still satisfied.
The Lie algebra of the automorphism group of $\mathbb{P}^n$ is $\mathfrak{g}_1=\mathfrak{sl}(n+1,\mathbb{C}) \subset \End(W)$ with $\dim_\mathbb{C}(W)=n+1$, and the holomorphic tangent space can be identified with 
\[\mathfrak{m}_1^+=\{\begin{bmatrix}
0 & A \\
0 & 0\\
\end{bmatrix}:A\ \text{is an $n\times 1$ matrix}\}\]
We considet the tangent vector corresponding to $A=(1,0,...,0)^t$, choose a basis $\{w_1,...w_{n+1}\}$ for $W$ then the tangent vector is corresponding to a linear transformation $T$ with $T(w_j)=0 (1\leq j\leq n), T(w_{n+1})=w_1$. Therefore the induced action of $T$ (still denoted by $T$) on $\bigwedge^mW$ satisfies
$
T(w_{i_1}\wedge w_{i_2} \wedge \cdots w_{i_m})=0
$ if $n+1\notin \{i_1,...,i_m\}$ or $1\in \{i_1,...,i_m\}$, thus the rank of $T\in End(\bigwedge^mW)$ is $\binom{n}{m-1}-\binom{n-1}{m-2}=\binom{n-1}{m-1}$. We can easily check that
\[\binom{n}{m-1}-\binom{n-1}{m-2} < \min\{\binom{n}{m-1}, \binom{n}{m}\}\]
the tangent vector on $G(p,q)$ is not of maximal rank and thus the embedding is not of diagonal type. 
\par To see whether $[V,[\mathfrak{m}^-,V]]$ will generate $\mathfrak{m}^+$ if $V=T_o(f(\mathbb{P}^n))$ where $f$ denotes the embedding above, we give the following easy lemma at first
\begin{lemma}\label{block span}
For any nonzero $w_{i_1}\wedge \cdots \wedge w_{i_{m-1}}\wedge w_{n+1}, w_{j_1}\wedge \cdots \wedge w_{j_{m}} \in \bigwedge^mW$ with $n+1\notin\{j_1,...,j_{m}\}$, there exists nonzero $T\in \End(W)$ with $T(w_j)=0$ for all $j\neq n+1$ and $T(w_{n+1})\in \text{span}\{w_1,...,w_n\}$ such that the induced action $T(w_{i_1}\wedge \cdots \wedge w_{i_{m-1}}\wedge w_{n+1})\neq 0$ and $T(W^{\wedge})=w_{j_1}\wedge \cdots \wedge w_{j_{m}}$ for some $W^{\wedge}\in \bigwedge^mW$.
\end{lemma}
\begin{proof}
From the given condition we know there exists some $j_k$ such that $j_k\notin \{i_1,...,i_{m-1}\}$, then choose $T_k$ which satisfies $T_{k}(w_{n+1})=w_{j_k}$ and $T_{k}(w_j)=0 (j\neq n+1)$, one can easily observe that $T_k$ is a required linear transformation. 
\end{proof}
We interpret Lemma \ref{block span} in term of matrices. Write the holomorphic tangent space of $M$ as 
\[T_o(M)\cong\mathfrak{m}^+=\{\begin{bmatrix}
0 & C \\
0 & 0\\
\end{bmatrix}:C\ \text{is a $p\times q$ matrix}\}\]
The matrix is written in terms of the basis $\{w_{i_1}\wedge \cdots w_{i_m}, 1\leq i_1<\cdots \leq i_m\leq n+1\}$ and the first $q$ rows are corresponding to the tranformation of $w_{i_1}\wedge \cdots w_{i_m}$ with $i_m=n+1$. Then the image of the $T_k$ above is corresponding to a matrix $C_k$ with only one $\binom{n-1}{m-1}\times \binom{n-1}{m-1}$ principal submatrix $C'_k$ of rank $\binom{n-1}{m-1}$ and from the lemma any position in $C$ can be covered by some of such $C'_k$'s. 
Therefore Lemma \ref{block span} says that the $(H_2)$-embedding $\mathbb{P}^m \subset G(p,q)$ satisfies \textbf{Condition C}. 
\par Now we are ready to finish the proof of Theorem \ref{generalaffineness}.

\begin{proof}
It suffices to check the Lie bracket generating condition $(\dag)$ for all $(H_2)$-subspaces. When the ambient space is of Type IV, $E_6$ or $E_7$, we have already done. It remains to consider the cases when $M$ is of Type I,II,III. We know the prototype to solve these problems comes from the diagonal curve in a tube type ambient space, comparing our situation with the prototype, we can observe that when $M=G(p,q)$ with $p\neq q$, the $(H_2)$-subspaces contain no Type II, III, IV factors and the 'trouble'  comes from the non-diagonal type embedding of a projective space. When $M=G(n,n), G^{II}(n,n), G^{III}(n,n)$, besides the non-diagonal type embedding of a projective space,  the 'trouble' also comes from Type II factor for example  $G^{II}(n,n)\subset G(n,n)$ with odd number $n$ and $G^{II}(r,r)\times G^{II}(n-r,n-r) \subset G^{II}(n,n)$ with odd number $r$ or odd number $n-r$, which are non-diagonal embeddings (but satisfy the \textbf{Condition C}).

\par Based on the observation we want to check the \textbf{Condition C} generally, we will  consider the problem using a chain of maximal $(H_2)$-subspaces. There are three types of maximal embeddings when reducible Hermitian symmetric space appears in the chain, the first type is $X_1 \times \cdots X_{i1}\times X_{i2} \cdots \times X_k \subset X_1 \times \cdots X_i\times \cdots \times X_k \subset M$ assuming all factors are irreducible, where $X_{i1}\times X_{i2}$ is a maximal $H_2$-subspace in $X_i$ and other factors are preserved; the second type is $ X_1\times \cdots X_i \times \cdots \times X_k \subset X_1\times \cdots X_i\times X_i \times \cdots \times X_k \subset M$ assuming all factors here are irreducible, where $X_i\subset X_i\times X_i$ is the diagonal embedding and also other factors are preserved; and the third type is $X_1 \times \cdots X_{i'}\cdots \times X_k \subset X_1 \times \cdots X_i\times \cdots \times X_k \subset M$ assuming all factors are irreducible, where $X_{i'}$ is an irreducible maximal $(H_2)$-subspace in $X_i$ and other factors are preserved. All $(H_2)$-subspaces can be obtained by a composition of these three types of embeddings.

\par From the table in \cite[pp.27-28]{MR2127948} we can find that all maximal $(H_2)$-subspaces in $M$ satisfy \textbf{Condition C}. We then use induction argument and suppose that $M'=X_1 \times \cdots X_i \times\cdots X_k\subset M$ is $H_2$ and satisfies \textbf{Condition C}. When $M'$ contains only Type I,II,III factors except for projective spaces, using block form together with the previous discussion we can easily observe that all these three types of maximal $(H_2)$-subspaces of $M'$ inherits \textbf{Condition C}. When $M'$ contains hyperquadric factor and $X_i=Q^s$ then from the table in \cite[pp.27-28]{MR2127948} we know $M$ must be $G(n,n), G^{II}(n,n), G^{III}(n,n)$ and the only maximal $(H_2)$-subspaces in $M'$ changing $X_i$ are $X_1 \times \cdots Q^{s} \times  \cdots \times X_k \subset X_1 \times \cdots Q^{s}\times Q^{s}\times  \cdots \times X_k=M' \subset M$ assuming there is another factor $Q^s$ in $M'$ such that $Q^s$ in the subspace is diagonally embedded in $Q^s\times Q^s$, or $X_1 \times \cdots Q^{s'} \times  \cdots \times X_k \subset X_1 \times \cdots Q^{s}\times  \cdots \times X_k=M' \subset M$ for some $s'<s$. Since the embedding of a hyperquadric as a maximal $(H_2)$-subspace is always of diagonal type in a tube type ambient space, this case can be easily done. If $X_i$ is preserved, then the case follows from the case when $M'$ contains only Type I,II,III factors directly. When $M'$ contains $X_i=\mathbb{P}^s$ factor, then only maximal $(H_2)$-subspaces in $M'$ changing $X_i$ (also if $X_i$ is preserved this follows from the cases discussed above) is $X_1 \times \cdots \mathbb{P}^{s} \times  \cdots \times X_k \subset X_1 \times \cdots \mathbb{P}^{s}\times \mathbb{P}^{s}\times  \cdots \times X_k \subset M$ assuming there is another factor $\mathbb{P}^s$ in $M'$ and $\mathbb{P}^s$ is diagonally embedded in $\mathbb{P}^s\times \mathbb{P}^s$ as there is no maximal $(H_2)$-subspace for a projective space, for this case we discuss more as follows.
\par  We need to give more explanations for cases concerning the diagonal embedding of projective space which is less obvious to preserve the \textbf{Condition C}. From the table in \cite[pp.27-28]{MR2127948}, if $M=G(p,q)$ with $p\neq q$, all factors must be some $G(p',q')$ with $\frac{p'}{q'}=\frac{p}{q}$ and projective spaces. In this case note that the embedding $\mathbb{P}^s \subset G(p',q')$ satisfies $p':q'=m:s-m+1$ for some $1<m<s$ and the factor $G(p',q')$ in an $(H_2)$-embedding satisfies $\frac{p}{q}=\frac{p'}{q'}$, so $m$ can be uniquely determined, we conclude that the diagonal embedding of projective space must be of the form $\mathbb{P}^s \subset \mathbb{P}^s \times \mathbb{P}^s \subset G(p',q')\times G(p',q')$. Therefore in this case we can observe that in a chain of $(H_2)$-subspaces, \textbf{Condition C} is always preserved. If $M=G(n,n), G^{II}(n,n), G^{III}(n,n)$, then in the chain of $(H_2)$-subspaces, embedding of projective space into Type II or Type III spaces can be uniquely determined by the dimension of the projective space and a $\mathbb{P}^s$ can not be a maximal $(H_2)$-subspace of Type II and Type III spaces at the same time, which implies that the chain of maximal $(H_2)$-subspace $\mathbb{P}^s\subset \mathbb{P}^s\times\mathbb{P}^s \subset X_1\times X_2$ with $X_1,X_2$ irreducible and $X_1\neq X_2$ will not appear. Hence \textbf{Condition C} is also preserved. 

\par Hence the (weaker) gap rigidity holds if there is a $P$-invariant Zariski open subset $\mathcal{O}_o=Gr(\dim(X),T_o(M))-\mathcal{Z}_o$ for some $P$-invariant hypersurface $\mathcal{Z}_o$ such that $[T_o(X)]\in \mathcal{O}_o$. 
\end{proof}

\begin{remark}
We did not find a conceptual proof on the Lie bracket generating condition $(\dag)$ independent of the classification and we can observe that $(H_2)$-condition is sufficient but not necessary, since for example $\mathbb{P}^2\times \mathbb{P}^2 \subset G(3,3)$ by standard embedding is a $(H_1)$-subspace satisfying the Lie bracket generating condition but it is not $H_2$.
\end{remark}
\section*{Acknowledgement}
The work was partially supported by a Postdoctoral International 
Exchange Program in China. The author would like to thank Professor Ngaiming Mok for helpful discussions. 


\bibliographystyle{alpha}  
\bibliography{research} 


%\begin{thebibliography}{00}



%\end{thebibliography}
\end{document}